\documentclass[12pt]{amsproc}

\usepackage{mathtools}
\usepackage{xcolor}
\usepackage[margin=1in]{geometry}
\usepackage[hypertexnames=false]{hyperref}
\usepackage{aliascnt}

\usepackage[capitalize,nameinlink,noabbrev]{cleveref}
\hypersetup{
	colorlinks=true,       
	linkcolor=purple,         
	citecolor=purple,        
	urlcolor=purple,           
}

\makeatletter
\patchcmd{\@maketitle}
  {\@settitle}
  {\vspace*{-2.5cm}\@settitle}
  {}{}
\makeatother

\theoremstyle{plain}
\newtheorem{theorem}{Theorem}

\newaliascnt{lemma}{theorem}
\newtheorem{lemma}[lemma]{Lemma}
\aliascntresetthe{lemma}

\newaliascnt{corollary}{theorem}
\newtheorem{corollary}[corollary]{Corollary}
\aliascntresetthe{corollary}

\newcommand{\bbQ}{\mathbb{Q}}
\newcommand{\bk}{\boldsymbol{k}}
\newcommand{\bl}{\boldsymbol{l}}

\begin{document}
\title{A note on Zlobin's note}

\author{Shin-ichiro Seki}
\address{Nagahama Institute of Bio-Science and Technology, 1266 Tamura-cho, Nagahama, Shiga, 526-0829, Japan}
\email{s\_seki@nagahama-i-bio.ac.jp}
\maketitle

Very little is known about the arithmetic nature, to use Zlobin's terminology, of multiple zeta values
\[
\zeta(k_1,\dots,k_r)=\sum_{0<n_1<\cdots<n_r}\frac{1}{n_1^{k_1}\cdots n_r^{k_r}},
\]
where $k_1, \dots, k_r$ are positive integers with $k_r\geq 2$.
The integer $k_1+\cdots+k_r$ is called the weight.
When the weight is even, say $2k$, there are several multiple zeta values that can be expressed as rational multiples of $\pi^{2k}$, and hence are transcendental by Lindemann's theorem.
In contrast, in odd weight, the only multiple zeta value currently known to be irrational is Ap\'ery's $\zeta(3)$, which also occurs as the double zeta value $\zeta(1,2)$.

By an odd zeta value, we mean a value of the Riemann zeta function $\zeta(s)$ at an odd integer $s\geq 3$.
It is known that there exist odd zeta values $\zeta(k)$ such that $1$, $\zeta(3)$, and $\zeta(k)$ are linearly independent over $\bbQ$.
For an odd integer $N\geq 5$, let $P(N)$ denote the following property:
\[
\text{there exists } k\in\{5,7,\dots,N\}
\text{ such that } 1,\ \zeta(3),\ \text{and } \zeta(k)
\text{ are linearly independent over } \mathbb{Q}.
\]
Ball and Rivoal \cite{BallRivoal} proved that $P(169)$ holds.
Fischler and Zudilin improved this result by proving that $P(139)$ holds, and Lai has announced that $P(75)$ holds.
See \cite[Theorem~3]{FischlerZudilin} and \cite[Claim~1.4]{Lai}.

On the other hand, Zlobin proved an analogous result in which odd zeta values are replaced by multiple zeta values of odd weight.
More precisely, Zlobin considered double zeta values of weight $5$ and triple zeta values of weight $7$ whose components are all either $2$ or $3$.
\begin{theorem}[{Zlobin~\cite[Corollary~3]{Zlobin}}]
There exists
\[
\bk\in\{(3,2), (2,3), (3,2,2), (2,3,2), (2,2,3)\}
\]
such that $1$, $\zeta(3)$, and $\zeta(\bk)$ are linearly independent over $\bbQ$.
\end{theorem}
Zlobin's proof does not proceed by constructing suitable linear forms and applying Nesterenko's criterion.
Instead, it is based on a simple observation that the linear independence of $1$, $\pi^2$, and $\pi^4$ over $\bbQ$ implies the linear independence over $\bbQ$ of at least one of the following two triples: $\bigl(1,\ \zeta(3),\ \zeta(3)\zeta(2)\bigr)$ and $\bigl(1,\ \zeta(3),\ \zeta(3)\zeta(4)\bigr)$.

In this note, we present the following slight extension of Zlobin's result.
\begin{theorem}\label{thm:Zlobin2}
There exist two distinct elements
\[
\bk,\bl\in\{(3,2),(2,3),(3,2,2),(2,3,2),(2,2,3)\}
\]
such that both $\zeta(\bk)$ and $\zeta(\bl)$ are transcendental and, moreover, $1$, $\zeta(3)$, and $\zeta(\bk)$ are linearly independent over $\bbQ$.
\end{theorem}
We now recall a result from Zagier's paper \cite{Zagier}.
For a nonnegative integer $n$ and an integer $i$ with $0\leq i\leq n$, set $z_{n,i}\coloneqq \zeta(\{2\}^i,3,\{2\}^{n-i})$ and $h_n\coloneqq \zeta(\{2\}^n)=\pi^{2n}/(2n+1)!$, where $\zeta(\{2\}^0)=h_0=1$.
Define the column vectors $Z_n$ and $R_n$, each of length $n+1$, by
\[
Z_n\coloneqq
\begin{pmatrix}
z_{n,0}\\
z_{n,1}\\
\vdots\\
z_{n,n}
\end{pmatrix},
\qquad
R_n\coloneqq
\begin{pmatrix}
h_n\zeta(3)\\
h_{n-1}\zeta(5)\\
\vdots\\
h_1\zeta(2n+1)\\
\zeta(2n+3)
\end{pmatrix}
\]
and define the matrix $A_n\coloneqq(a_{ij})_{0\leq i\leq n, 1\leq j\leq n+1}\in M_{n+1}(\bbQ)$ by
\[
a_{ij}\coloneqq 2(-1)^j\left[\binom{2j}{2i+2}-\left(1-\frac{1}{2^{2j}}\right)\binom{2j}{2(n-i)+1}\right].
\]
Zagier's theorem \cite[Theorem~1]{Zagier} asserts that
\[
Z_n=A_nR_n
\]
and that $A_n$ is invertible, or equivalently, that $A_n\in GL_{n+1}(\bbQ)$.
Set $d_n\coloneqq\frac{(n+1)(2n+3)}{3}$.
Then $\zeta(2)h_n=d_nh_{n+1}$.
Therefore, if we define a matrix $D_n\in M_{n+1,n+2}(\bbQ)$ by
\[
D_n\coloneqq
\begin{pmatrix}
d_n    & 0       & \cdots & \cdots & 0   & 0\\
0      & d_{n-1} & \ddots &        & \vdots & \vdots\\
\vdots & \ddots  & \ddots & \ddots & \vdots & \vdots\\
\vdots &         & \ddots & d_1    & 0   & \vdots\\
0      & \cdots  & \cdots & 0      & 1   & 0
\end{pmatrix},
\]
then $\zeta(2)R_n=D_nR_{n+1}$.
It follows from Zagier's theorem that
\[
\zeta(2)Z_n=\zeta(2)A_nR_n=A_nD_nR_{n+1}=A_nD_nA_{n+1}^{-1}Z_{n+1}.
\]
Hence, on setting $B_n\coloneqq A_nD_nA_{n+1}^{-1}\in M_{n+1,n+2}(\bbQ)$, we obtain
\begin{equation}\label{eq:key}
\zeta(2)Z_n=B_nZ_{n+1}.
\end{equation}
\begin{proof}[Proof of \cref{thm:Zlobin2}]
Since 
\[
A_0=(1),\qquad A_1=\frac{1}{2}\begin{pmatrix} -4 & 9 \\ 6 & -11\end{pmatrix}, \qquad A_2=\frac{1}{16}
\begin{pmatrix}
-32 & 192 & -291\\
0 & -88 & 150\\
48 & -120 & 157
\end{pmatrix},
\]
we have
\[
B_0=\frac{1}{5}\begin{pmatrix} 11 & 9\end{pmatrix},\qquad B_1=\frac{1}{453}
\begin{pmatrix}
1690 & 3153 & 120 \\
840 & -537 & 2070
\end{pmatrix}.
\]
Hence, by \eqref{eq:key}, we have
\begin{align}
5\zeta(2)\zeta(3)&=11\zeta(3,2)+9\zeta(2,3),\label{eq:2-3}\\
453\zeta(2)\zeta(3,2)&=1690\zeta(3,2,2)+3153\zeta(2,3,2)+120\zeta(2,2,3),\label{eq:2-32}\\
151\zeta(2)\zeta(2,3)&=280\zeta(3,2,2)-179\zeta(2,3,2)+690\zeta(2,2,3).\label{eq:2-23}
\end{align}
Let $S\coloneqq\{\zeta(3,2),\zeta(2,3),\zeta(3,2,2),\zeta(2,3,2),\zeta(2,2,3)\}$.
We first prove that at least one element of $S$ does not belong to $\overline{\bbQ}+\overline{\bbQ}\zeta(3)$.
Suppose, for contradiction, that $S\subset\overline{\bbQ}+\overline{\bbQ}\zeta(3)$.
Then, by \eqref{eq:2-3}, \eqref{eq:2-32}, and \eqref{eq:2-23}, we have $\zeta(2)\zeta(3),\zeta(2)\zeta(3,2),\zeta(2)\zeta(2,3)\in\overline{\bbQ}+\overline{\bbQ}\zeta(3)$.
Suppose first that $\zeta(3)\in\overline{\bbQ}$.
Since $\zeta(2)\zeta(3)\in\overline{\bbQ}$, this would imply that $\zeta(2)\in\overline{\bbQ}$, contradicting the transcendence of $\zeta(2)$.
We may therefore assume that $\zeta(3)\notin\overline{\bbQ}$.
Write $\zeta(2)\zeta(3)=a+b\zeta(3)$, $\zeta(3,2)=c+d\zeta(3)$, and $\zeta(2,3)=e+f\zeta(3)$, where $a,b,c,d,e,f\in\overline{\bbQ}$.
The transcendence of $\zeta(2)$ implies that $a\neq 0$.
We have
\[
\zeta(2)\zeta(3,2)=\left(\frac{a}{\zeta(3)}+b\right)(c+d\zeta(3))=\frac{ac}{\zeta(3)}+ad+bc+bd\zeta(3)\in\overline{\bbQ}+\overline{\bbQ}\zeta(3).
\]
If $c\neq 0$, it follows that $\frac{1}{\zeta(3)}\in\overline{\bbQ}+\overline{\bbQ}\zeta(3)$.
Hence $\zeta(3)$ is a root of a polynomial of degree at most $2$ with coefficients in $\overline{\bbQ}$.
This contradicts the assumption that $\zeta(3)\notin\overline{\bbQ}$. Therefore, $c=0$.
The same argument shows that $e=0$.
Consequently, $\zeta(3,2)=d\zeta(3)$ and $\zeta(2,3)=f\zeta(3)$, and hence \eqref{eq:2-3} gives $\zeta(2)=\frac{11}{5}d+\frac{9}{5}f\in\overline{\bbQ}$. This is a contradiction.

We may therefore choose $x\in S$ such that $x\notin\overline{\bbQ}+\overline{\bbQ}\zeta(3)$.
Since $\zeta(3)$ is irrational, the condition $x\notin\overline{\bbQ}+\overline{\bbQ}\zeta(3)$ implies that $1$, $\zeta(3)$, and $x$ are linearly independent over $\bbQ$.
If $x=\zeta(3,2)$, then \eqref{eq:2-23}, together with the transcendence of $\zeta(2)$, shows that at least one of $\zeta(2,3)$, $\zeta(3,2,2)$, $\zeta(2,3,2)$, and $\zeta(2,2,3)$ is transcendental.
Similarly, if $x=\zeta(2,3)$, we use \eqref{eq:2-32}.
If $x=\zeta(3,2,2)$, we use the relation
\[
84\zeta(2)\zeta(3,2)-169\zeta(2)\zeta(2,3)=785\zeta(2,3,2)-750\zeta(2,2,3)>0,
\]
which is obtained by eliminating $\zeta(3,2,2)$ from
\eqref{eq:2-32} and \eqref{eq:2-23}. 
If $x=\zeta(2,3,2)$, we use the relation
\[
537\zeta(2)\zeta(3,2)+3153\zeta(2)\zeta(2,3)=7850\zeta(3,2,2)+14550\zeta(2,2,3),
\]
which is obtained by eliminating $\zeta(2,3,2)$.
Finally, if $x=\zeta(2,2,3)$, we use the relation
\[
69\zeta(2)\zeta(3,2)-4\zeta(2)\zeta(2,3)=250\zeta(3,2,2)+485\zeta(2,3,2)>0,
\]
which is obtained by eliminating $\zeta(2,2,3)$.
In each case, these relations imply the existence of another transcendental number.
\end{proof}
\begin{lemma}\label{lem:general}
Let $n$ be a nonnegative integer, and set
\[
C_n\coloneqq [I_{n+1}\mid -B_n]=(v_1,\dots,v_{2n+3})\in M_{n+1,2n+3}(\bbQ).
\]
Suppose that, for every choice of integers $1\leq i_1<\cdots<i_n\leq 2n+3$, we have
\[
l_{i_1,\dots,i_n}\coloneqq\sum_{k=0}^n\det(v_{i_1},\dots,v_{i_n},v_{k+1})z_{n,k}\neq 0.
\]
When $n=0$, the tuple $(i_1,\dots,i_n)$ is understood to be empty.
Then at least $n+1$ of the following $2n+3$ numbers are transcendental:
\[
z_{n,0},\dots,z_{n,n}, \quad z_{n+1,0},\dots,z_{n+1,n+1}.
\]
\end{lemma}
\begin{proof}
Denote these $2n+3$ numbers, in the order listed, by $z_1,\dots,z_{2n+3}$. Thus,
\[
z_{k+1}=z_{n,k}\quad (0\leq k\leq n),\qquad z_{n+k+2}=z_{n+1,k}\quad (0\leq k\leq n+1).
\]
Suppose, for contradiction, that at most $n$ of these numbers are transcendental, and choose integers $1\leq i_1<\cdots<i_n\leq 2n+3$ such that all the transcendental numbers among $z_1,\dots,z_{2n+3}$ are contained in $\{z_{i_1},\dots,z_{i_n}\}$.

There exist scalars $c_1,\dots,c_{n+1}\in\bbQ$ such that
\[
\det(v_{i_1},\dots,v_{i_n},v)=(c_1,\dots, c_{n+1})v
\]
for every column vector $v$ of length $n+1$.
This follows by expanding the determinant along its last column.
It follows from \eqref{eq:key} that
\[
C_n
\begin{pmatrix}
\zeta(2)z_1\\
\vdots\\
\zeta(2)z_{n+1}\\
z_{n+2}\\
\vdots\\
z_{2n+3}
\end{pmatrix}=\sum_{k=1}^{n+1}v_k\zeta(2)z_k+\sum_{k=n+2}^{2n+3}v_kz_k=\boldsymbol{0}.
\]
Multiplying this identity on the left by $(c_1,\dots,c_{n+1})$, we obtain
\[
\sum_{k=1}^{n+1}\det(v_{i_1},\dots, v_{i_n},v_k)\zeta(2)z_k+\sum_{k=n+2}^{2n+3}\det(v_{i_1},\dots, v_{i_n},v_k)z_k=0.
\]
The terms corresponding to $k\in\{i_1,\dots,i_n\}$ vanish.
Therefore, since $z_j$ is algebraic for every
\[
j\in\{1,\dots,2n+3\}\setminus\{i_1,\dots,i_n\},
\]
and since the denominator below is nonzero by assumption, we obtain
\[
\zeta(2)
=
\frac{
-\displaystyle\sum_{k=0}^{n+1}
\det(v_{i_1},\dots,v_{i_n},v_{n+k+2})z_{n+1,k}
}{
l_{i_1,\dots,i_n}
}
\in\overline{\bbQ}.
\]
This contradicts the transcendence of $\zeta(2)$.
\end{proof}
For each $1\leq i_1<\cdots<i_n\leq 2n+3$, let $\widetilde l_{i_1,\dots,i_n}$ be the nonzero rational multiple of $l_{i_1,\dots,i_n}$ whose coefficients are relatively prime integers and whose first nonzero coefficient is positive.
\begin{corollary}
At least three of the following seven multiple zeta values are transcendental:
\[
\begin{gathered}
\zeta(3,2,2),\quad
\zeta(2,3,2),\quad
\zeta(2,2,3),\\
\zeta(3,2,2,2),\quad
\zeta(2,3,2,2),\quad
\zeta(2,2,3,2),\quad
\zeta(2,2,2,3).
\end{gathered}
\]
\end{corollary}
\begin{proof}
A direct calculation gives
\[
A_3=
\frac{1}{16}
\begin{pmatrix}
-32 & 192 & -480 & 641\\
0 & 32 & -291 & 455\\
0 & -120 & 598 & -889\\
48 & -120 & 189 & -223
\end{pmatrix}
\]
and
\[
B_2=
\frac{1}{2085}
\begin{pmatrix}
16269 & 9561 & 16344 & 1116\\
-840 & 19490 & 9510 & -560\\
6048 & -6888 & -3837 & 18627
\end{pmatrix}.
\]
For the $21$ choices of integers $i$ and $j$ satisfying $1\leq i<j\leq 7$, all the quantities $\widetilde l_{i,j}$ can be verified by direct computation to be nonzero, as shown below.
The conclusion therefore follows from \cref{lem:general}.
Note that, even when negative coefficients occur, the sign of $\widetilde l_{i,j}$ can sometimes be determined easily by using the inequalities $z_{2,0}>z_{2,1}>z_{2,2}$, which follow immediately from the definition.
\begin{align*}
\widetilde l_{1,2}&=z_{2,2}>0\\
\widetilde l_{1,3}&=z_{2,1}>0\\
\widetilde l_{1,4}&=36z_{2,1} + 5z_{2,2}>0\\
\widetilde l_{1,5}&=3444z_{2,1} + 9745z_{2,2}>0\\
\widetilde l_{1,6}&=1279z_{2,1} + 3170z_{2,2}>0\\
\widetilde l_{1,7}&=2661z_{2,1} + 80z_{2,2}>0\\
\widetilde l_{2,3}&=z_{2,0}>0\\
\widetilde l_{2,4}&=2016z_{2,0} - 5423z_{2,2}=199.4075\ldots>0\\
\widetilde l_{2,5}&=2296z_{2,0} + 3187z_{2,2}>0\\
\widetilde l_{2,6}&=1279z_{2,0} + 5448z_{2,2}>0\\
\widetilde l_{2,7}&=6209z_{2,0} - 372z_{2,2}>0\\
\widetilde l_{3,4}&=280z_{2,0} + 5423z_{2,1}>0\\
\widetilde l_{3,5}&=19490z_{2,0} - 9561z_{2,1}>0\\
\widetilde l_{3,6}&=1585z_{2,0} - 2724z_{2,1}=84.7017\ldots>0\\
\widetilde l_{3,7}&=140z_{2,0} + 279z_{2,1}>0\\
\widetilde l_{4,5}&=5376z_{2,0} - 8148z_{2,1} - 15593z_{2,2}=-88.2430\ldots<0\\
\widetilde l_{4,6}&=8680z_{2,0} - 25783z_{2,1} - 26930z_{2,2}=-1103.3161\ldots<0\\
\widetilde l_{4,7}&=840z_{2,0} + 20301z_{2,1} + 560z_{2,2}>0\\
\widetilde l_{5,6}&=4450z_{2,0} + 36399z_{2,1} + 109170z_{2,2}>0\\
\widetilde l_{5,7}&=172270z_{2,0} - 89103z_{2,1} - 13000z_{2,2}>0\\
\widetilde l_{6,7}&=41965z_{2,0} - 74034z_{2,1} - 4740z_{2,2}=1966.7522\ldots>0.\qedhere
\end{align*}
\end{proof}
\begin{corollary}
At least four of the following nine multiple zeta values are transcendental:
\[
\begin{gathered}
\zeta(3,2,2,2),\quad
\zeta(2,3,2,2),\quad
\zeta(2,2,3,2),\quad
\zeta(2,2,2,3),\\
\zeta(3,2,2,2,2),\quad
\zeta(2,3,2,2,2),\quad
\zeta(2,2,3,2,2),\quad
\zeta(2,2,2,3,2),\quad
\zeta(2,2,2,2,3).
\end{gathered}
\]
\end{corollary}
\begin{proof}
A direct calculation gives
\[
A_4=
\frac{1}{256}
\begin{pmatrix}
-512 & 3072 & -7680 & 14336 & -17925\\
0 & 512 & -7680 & 31760 & -46140\\
0 & 0 & 2512 & -14224 & 21378\\
0 & -1920 & 10080 & -28048 & 38340\\
768 & -1920 & 3024 & -4080 & 4603
\end{pmatrix}
\]
and
\[
B_3=
\frac{1}{7665513}
\begin{pmatrix}
97968276 & 42880725 & 33333240 & 78342276 & 3988080\\
3492720 & 56607621 & 117346350 & 4046280 & 2328480\\
-7056000 & 59083080 & 9936530 & 62828479 & -4704000\\
30796416 & -29746656 & -31191480 & -17516685 & 112517100
\end{pmatrix}.
\]
For the $84$ choices of integers $i$, $j$, and $k$ satisfying $1\leq i<j<k\leq 9$, all the quantities $\widetilde l_{i,j,k}$ can be verified by direct computation to be nonzero, as shown below.
The conclusion therefore follows from \cref{lem:general}.
Note that $z_{3,0}>z_{3,1}>z_{3,2}>z_{3,3}$.
\begin{align*}
\widetilde l_{1,2,3}&=z_{3,3}>0\\
\widetilde l_{1,2,4}&=z_{3,2}>0\\
\widetilde l_{1,2,5}&=3819z_{3,2} + 875z_{3,3}>0\\
\widetilde l_{1,2,6}&=1239444z_{3,2} + 2461795z_{3,3}>0\\
\widetilde l_{1,2,7}&=3119148z_{3,2} + 993653z_{3,3}>0\\
\widetilde l_{1,2,8}&=761595z_{3,2} + 2731673z_{3,3}>0\\
\widetilde l_{1,2,9}&=375057z_{3,2} + 15680z_{3,3}>0\\
\widetilde l_{1,3,4}&=z_{3,1}>0\\
\widetilde l_{1,3,5}&=10184z_{3,1} - 1155z_{3,3}>0\\
\widetilde l_{1,3,6}&=9915552z_{3,1} + 18869207z_{3,3}>0\\
\widetilde l_{1,3,7}&=1039716z_{3,1} + 3911545z_{3,3}>0\\
\widetilde l_{1,3,8}&=1167779z_{3,1} + 269752z_{3,3}>0\\
\widetilde l_{1,3,9}&=208365z_{3,1} - 4312z_{3,3}>0\\
\widetilde l_{1,4,5}&=200z_{3,1} + 99z_{3,2}>0\\
\widetilde l_{1,4,6}&=2813480z_{3,1} - 2695601z_{3,2}>0\\
\widetilde l_{1,4,7}&=993653z_{3,1} - 11734635z_{3,2}=-56416.2857\ldots<0\\
\widetilde l_{1,4,8}&=8975497z_{3,1} - 578040z_{3,2}>0\\
\widetilde l_{1,4,9}&=200z_{3,1} + 99z_{3,2}>0\\
\widetilde l_{1,5,6}&=3720z_{3,1} - 4269z_{3,2} - 1400z_{3,3}=12.9719\ldots>0\\
\widetilde l_{1,5,7}&=1112z_{3,1} - 48180z_{3,2} - 11165z_{3,3}=-292.5035\ldots<0\\
\widetilde l_{1,5,8}&=33488z_{3,1} - 3435z_{3,2} - 4585z_{3,3}>0\\
\widetilde l_{1,5,9}&=200z_{3,1} + 99z_{3,2}>0\\
\end{align*}

\begin{align*}
\widetilde l_{1,6,7}&=20185344z_{3,1} + 22503204z_{3,2} + 83108599z_{3,3}>0\\
\widetilde l_{1,6,8}&=15542664z_{3,1} + 16236255z_{3,2} + 61826149z_{3,3}>0\\
\widetilde l_{1,6,9}&=2021400z_{3,1} - 1999863z_{3,2} - 125440z_{3,3}=10487.0516\ldots>0\\
\widetilde l_{1,7,8}&=23294699z_{3,1} + 25168695z_{3,2} + 95655525z_{3,3}>0\\
\widetilde l_{1,7,9}&=211185z_{3,1} - 2886543z_{3,2} - 125048z_{3,3}=-14559.2965\ldots<0\\
\widetilde l_{1,8,9}&=2170165z_{3,1} - 154080z_{3,2} - 51352z_{3,3}>0\\
\widetilde l_{2,3,4}&=z_{3,0}>0\\
\widetilde l_{2,3,5}&=122208z_{3,0} - 388763z_{3,3}=2068.6985\ldots>0\\
\widetilde l_{2,3,6}&=1101728z_{3,0} + 1588175z_{3,3}>0\\
\widetilde l_{2,3,7}&=259929z_{3,0} + 277777z_{3,3}>0\\
\widetilde l_{2,3,8}&=5838895z_{3,0} + 26114092z_{3,3}>0\\
\widetilde l_{2,3,9}&=21555z_{3,0} - 764z_{3,3}>0\\
\widetilde l_{2,4,5}&=28000z_{3,0} + 388763z_{3,2}>0\\
\widetilde l_{2,4,6}&=3938872z_{3,0} - 2858715z_{3,2}>0\\
\widetilde l_{2,4,7}&=993653z_{3,0} - 3333324z_{3,2}=5653.2117\ldots>0\\
\widetilde l_{2,4,8}&=62828479z_{3,0} - 78342276z_{3,2}=1117718.4012\ldots>0\\
\widetilde l_{2,4,9}&=19600z_{3,0} + 16617z_{3,2}>0\\
\widetilde l_{2,5,6}&=104160z_{3,0} - 274032z_{3,2} - 394135z_{3,3}=31.6059\ldots>0\\
\widetilde l_{2,5,7}&=13344z_{3,0} - 633996z_{3,2} - 187709z_{3,3}=-3780.2765\ldots<0\\
\widetilde l_{2,5,8}&=2812992z_{3,0} - 6412053z_{3,2} - 10417687z_{3,3}=7100.2988\ldots>0\\
\widetilde l_{2,5,9}&=16800z_{3,0} + 282141z_{3,2} + 11200z_{3,3}>0\\
\widetilde l_{2,6,7}&=6728448z_{3,0} - 1504404z_{3,2} + 6711205z_{3,3}>0\\
\widetilde l_{2,6,8}&=5180888z_{3,0} - 9810747z_{3,2} - 12017785z_{3,3}=42629.3849\ldots>0\\
\widetilde l_{2,6,9}&=673800z_{3,0} - 511821z_{3,2} - 45280z_{3,3}>0\\
\widetilde l_{2,7,8}&=23294699z_{3,0} - 24260916z_{3,2} + 17165536z_{3,3}=485170.7781\ldots>0\\
\widetilde l_{2,7,9}&=70395z_{3,0} - 280836z_{3,2} - 14236z_{3,3}=108.0429\ldots>0\\
\widetilde l_{2,8,9}&=2170165z_{3,0} - 2759640z_{3,2} - 192292z_{3,3}=37814.5789\ldots>0\\
\widetilde l_{3,4,5}&=13860z_{3,0} - 388763z_{3,1}=-3940.1310\ldots<0\\
\widetilde l_{3,4,6}&=18869207z_{3,0} - 14293575z_{3,1}>0\\
\widetilde l_{3,4,7}&=3911545z_{3,0} - 1111108z_{3,1}>0\\
\widetilde l_{3,4,8}&=337190z_{3,0} - 6528523z_{3,1}=-63561.2757\ldots<0\\
\end{align*}

\begin{align*}
\widetilde l_{3,4,9}&=3234z_{3,0} - 5539z_{3,1}=19.6197\ldots>0\\
\widetilde l_{3,5,6}&=956256z_{3,0} - 2192256z_{3,1} - 2793371z_{3,3}=-7354.4376\ldots<0\\
\widetilde l_{3,5,7}&=17520z_{3,0} - 19212z_{3,1} - 53555z_{3,3}=90.2641\ldots>0\\
\widetilde l_{3,5,8}&=96180z_{3,0} - 2137351z_{3,1} - 63560z_{3,3}=-21323.9339\ldots<0\\
\widetilde l_{3,5,9}&=924z_{3,0} - 31349z_{3,1} + 616z_{3,3}=-321.0014\ldots<0\\
\widetilde l_{3,6,7}&=7501068z_{3,0} + 1504404z_{3,1} + 13675889z_{3,3}>0\\
\widetilde l_{3,6,8}&=5412085z_{3,0} + 9810747z_{3,1} + 26471452z_{3,3}>0\\
\widetilde l_{3,6,9}&=666621z_{3,0} - 511821z_{3,1} - 13036z_{3,3}>0\\
\widetilde l_{3,7,8}&=8389565z_{3,0} + 8086972z_{3,1} + 39389860z_{3,3}>0\\
\widetilde l_{3,7,9}&=962181z_{3,0} - 280836z_{3,1} - 28292z_{3,3}>0\\
\widetilde l_{3,8,9}&=6420z_{3,0} - 114985z_{3,1} + 2152z_{3,3}=-1102.0557\ldots<0\\
\widetilde l_{4,5,6}&=44800z_{3,0} - 450440z_{3,1} + 399053z_{3,2}=-1557.1322\ldots<0\\
\widetilde l_{4,5,7}&=133980z_{3,0} - 187709z_{3,1} + 1767315z_{3,2}=11419.7430\ldots>0\\
\widetilde l_{4,5,8}&=55020z_{3,0} - 1488241z_{3,1} + 27240z_{3,2}=-14878.0539\ldots<0\\
\widetilde l_{4,5,9}&=200z_{3,1} + 99z_{3,2}>0\\
\widetilde l_{4,6,7}&=83108599z_{3,0} - 20133615z_{3,1} - 41027667z_{3,2}>0\\
\widetilde l_{4,6,8}&=61826149z_{3,0} + 36053355z_{3,1} - 79414356z_{3,2}>0\\
\widetilde l_{4,6,9}&=31360z_{3,0} - 33960z_{3,1} + 9777z_{3,2}=464.1344\ldots>0\\
\widetilde l_{4,7,8}&=95655525z_{3,0} - 17165536z_{3,1} - 118169580z_{3,2}=1518816.6557\ldots>0\\
\widetilde l_{4,7,9}&=31262z_{3,0} - 10677z_{3,1} + 21219z_{3,2}>0\\
\widetilde l_{4,8,9}&=12838z_{3,0} - 48073z_{3,1} - 12912z_{3,2}=-283.8185\ldots<0\\
\widetilde l_{5,6,7}&=3837792z_{3,0} - 4083552z_{3,1} - 5410548z_{3,2} - 12985147z_{3,3}=-13075.2001\ldots<0\\
\widetilde l_{5,6,8}&=2863392z_{3,0} - 962952z_{3,1} - 6428163z_{3,2} - 10472497z_{3,3}=-2489.0461\ldots<0\\
\widetilde l_{5,6,9}&=5376z_{3,0} - 63576z_{3,1} + 58815z_{3,2} + 3584z_{3,3}=-220.0641\ldots<0\\
\widetilde l_{5,7,8}&=885108z_{3,0} - 926215z_{3,1} - 1922547z_{3,2} - 3151113z_{3,3}=-7111.7603\ldots<0\\
\widetilde l_{5,7,9}&=26796z_{3,0} - 39321z_{3,1} + 430551z_{3,2} + 17864z_{3,3}=2751.9542\ldots>0\\
\widetilde l_{5,8,9}&=11004z_{3,0} - 351229z_{3,1} + 10944z_{3,2} + 7336z_{3,3}=-3516.2618\ldots<0\\
\widetilde l_{6,7,8}&=740431z_{3,0} - 17613831z_{3,1} - 19801959z_{3,2} - 71782441z_{3,3}=-470393.8674\ldots<0\\
\widetilde l_{6,7,9}&=962079z_{3,0} - 245079z_{3,1} - 488331z_{3,2} - 49444z_{3,3}>0\\
\widetilde l_{6,8,9}&=2148247z_{3,0} + 1197753z_{3,1} - 2816808z_{3,2} - 218692z_{3,3}>0\\
\widetilde l_{7,8,9}&=3323087z_{3,0} - 664212z_{3,1} - 4178568z_{3,2} - 278732z_{3,3}=50891.7185\ldots>0.\qedhere
\end{align*}
\end{proof}
Taking $n=0$, we see that at least one of $\zeta(3), \zeta(3,2), \zeta(2,3)$ is transcendental.
Moreover, using $\zeta(2,3)=3\zeta(2)\zeta(3)-\frac{11}{2}\zeta(5)$, we also find that at least one of $\zeta(3), \zeta(5), \zeta(2,3)$ is transcendental.
Although an inspection of the proof makes it clear that the latter statement is far more trivial than Zudilin's celebrated result \cite{Zudilin} that at least one of $\zeta(5), \zeta(7), \zeta(9), \zeta(11)$ is irrational, the formal resemblance between the two statements is at least superficially interesting.

\end{document}